\documentclass[11pt]{amsart} 
\usepackage{amssymb} 
\usepackage[utf8]{inputenc}
\usepackage[english]{babel}
\usepackage{lineno,hyperref}
\usepackage{amsmath}
\usepackage{amscd}
\usepackage{amsthm}
\usepackage{amsfonts}
\usepackage{color}
\usepackage{xfrac}
\usepackage{mathtools}
\usepackage{graphicx}
\usepackage{diagrams} 
\usepackage{tikz}
\usepackage{bbold}
\usepackage[T1]{fontenc}
\newtheorem{theorem}{Theorem}[section]

\newtheorem{corollary}[theorem]{Corollary}
\newtheorem{lemma}[theorem]{Lemma}
\newtheorem{remark}[theorem]{Remark}

\newcommand\restr[2]{\ensuremath{\left.#1\right|_{#2}}}
\newcommand*\rfrac[2]{{}^{#1}\!/_{#2}}

\begin{document}  
\title[On Steenrod $\mathbb{L}$-homology]{ 
On Steenrod $\mathbb{L}$-homology, generalized manifolds, and surgery} 
\author[F. Hegenbarth and D. Repov\v{s}]{ 
Friedrich Hegenbarth and Du\v{s}an Repov\v{s}}
\address{Dipartimento di Matematica "Federigo Enriques", 
Universit\` a degli studi di Milano, 
20133 Milano, Italy}
\email{friedrich.hegenbarth@unimi.it} 
\address{Faculty of Education and Faculty of Mathematics and Physics, University of Ljubljana \&
Institute of Mathematics, Physics and Mechanics, 
1000 Ljubljana, Slovenia}
\email{dusan.repovs@guest.arnes.si}
\begin{abstract}   
  The aim of this paper
is to show the importance
of the Steenrod construction of homology theories
for the disassembly process in surgery on a
 {\it generalized} $n$-manifold $X^n$,
in order to produce
 an element of  generalized homology theory, which is basic for calculations.
 In particular, 
 we show how to construct an element of the $n$-th Steenrod homology group $H^{st}_{n} (X^{n}, \mathbb{L}^+),$  
where $\mathbb{L}^+$
is the connected covering spectrum of the periodic surgery spectrum $\mathbb{L}$,
avoiding the use of the {\it geometric} splitting procedure, which is standardly used in surgery on {\it topological} manifolds.
\end{abstract} 
\dedicatory{Dedicated to the memory of Professor Andrew Ranicki (1948-2018)} 
\maketitle
\section{Introduction}\label{S:Intro} 
In order to study global objects it can be useful to decompose them into similar smaller pieces.
 This process of {\it disassembly} also applies to surgery theory. 
 If one does it in an appropriate way, it produces an element of a generalized homology theory, which is basic for calculations. 
 Here, ``appropriate'' means ``semisimplicially'' defined spectra (this holds for all spectra considered in the paper).
 
 Geometrically, one uses transversality to attain the goal.
  This works well 
    for 
   {\it PL topological}
   manifolds,
    but it does not work
   for {\it generalized} manifolds. 
  The aim of this paper
is to show that for generalized manifolds
an appropriate tool
 to overcome this problem is
the Steenrod construction of homology theory.

Steenrod homology is a homology theory which is
very
appropriate for compact metric spaces which have certain bad local
properties.
Generalized Steenrod homology theory has been well presented  by 
Ferry~\cite{Ferry95}
(for more cf. also Milnor~\cite{Milnor95}). 
A rigorous development of this theory was given earlier
by
Kahn, Kaminker, and Schochet~\cite{KaKaSc77}. 

The underlying spectra of 
 homology theory which we
 shall consider are $\mathbb{\Omega}^{N}$, $\mathbb{\Omega}^{PD}$, $\mathbb{\Omega}^{NPD}$, and $\mathbb{L}$.
They are defined simplicially, in terms of adic objects
 (cf.  
 Nicas~\cite{Nicas82},
  Quinn~\cite{Quinn69},
  and
   Ranicki~\cite{Ranicki92}). 
   Objects in $\mathbb{\Omega}^{N}$ (resp. $\mathbb{\Omega}^{PD}$) are adic normal spaces (resp. adic Poincar\'{e} duality complexes), and objects in $\mathbb{\Omega}^{NPD}$ are adic normal spaces with boundaries being adic Poincar\'{e} duality complexes
    (cf. Quinn~\cite{Quinn72}). 
 
 Our main interest will be the periodic surgery spectrum $\mathbb{L}$ with 
$$\mathbb{L}_0 = \mathbb{Z} \times \rfrac{G}{TOP}$$
 and its connected covering spectrum $\mathbb{L}\left<1\right>$, which we shall denote by $\mathbb{L}^+$. Elements of
  $\mathbb{L}^+$ are adic surgery problems (cf. Nicas~\cite{Nicas82}), and there is a fiber sequence of spectra 
 $$\mathbb{L}^+ \rightarrow \mathbb{L} \rightarrow \mathbb{K}(\mathbb{Z}, 0),$$ where
  $\mathbb{K}(\mathbb{Z}, 0)$
  is the Eilenberg-MacLane spectrum.

Steenrod homology is defined on compact metric spaces $X$, 
and we write $H^{st}_* (X, \mathbb{\mathcal{S}})$, where $\mathbb{\mathcal{S}}$ is any one of the above spectra.
 If $X$ is a PL topological manifold,
  then $H^{st}_* (X, \mathbb{\mathcal{S}})$
   coincides with ordinarily defined generalized homology
    $H_* (X, \mathbb{\mathcal{S}})$.

It is important to note that   $\mathbb{L}^+$
 (resp.  $\mathbb{L}$)
 can be defined algebraically and that the following holds.
\begin{theorem}[Ranicki~\cite{Ranicki79,Ranicki92}]
 \label{F:spectra}
There is a map of spectra 
$$\mathbb{\Omega}^{NPD} \rightarrow \Sigma \mathbb{L}^+,$$
where $\Sigma \mathbb{L}^+$ is the suspension spectrum of $\mathbb{L}^+$ 
$($cf. Ranicki~\cite[p.287]{Ranicki79}$)$. 
Moreover, the induced morphism 
$$H_n (K, \mathbb{\Omega}^{NPD}) \rightarrow H_{n-1} (K, \mathbb{L}^+)$$
 is an isomorphism for $n \geq 4$, where $K$ 
 is a finite polyhedron $($cf. 
 Hausmann and Vogel~\cite{HaVo93},
 Jones~\cite{Jones73},
  Levitt~\cite{Levitt72},
 and
Quinn~\cite{Quinn72}$)$.
\end{theorem}
Steenrod homology is related to locally finite homology.
\begin{theorem}[Ferry~\cite{Ferry95}, Milnor~\cite{Milnor95}]
\label{F:metricPair}
For every compact metric pair $(X, X')$,
 the natural homomorphism 
$$H^{st}_* (X, X', \mathbb{\mathcal{S}}) \rightarrow H^{lf}_* (X \setminus X', \mathbb{\mathcal{S}})$$
 is an isomorphism.
\end{theorem}
We shall apply this property only for $\mathbb{\mathcal{S}} = \mathbb{L}^+$.
The definition of $H^{lf}_* (\cdot, \mathbb{L}^+)$ can be found in
Ranicki~\cite[Appendix C]{Ranicki92}.

In order to verify the axioms of Steenrod homology theory, one has to use the following result.
\begin{theorem}[Ferry~\cite{Ferry95}, Milnor~\cite{Milnor95}]
\label{F:compact}
Any compact metric pair $(X, X')$ can be embedded into
 a compact metric pair $(T, T')$ so that
\begin{enumerate}
	\item $T$ and $T'$ are contractible;
	\item $T \setminus X$ is a CW-complex
	and 
	$T' \setminus X' \subset T \setminus X$ is a subcomplex.
\end{enumerate}
Moreover, the construction of $(T, T')$ is natural with respect to maps
between compact pairs
 $(X, X') \rightarrow (Y, Y')$.
\end{theorem}
We shall
adopt the notation from
Ferry~\cite{Ferry95} and write $T \setminus X = \text{OFC}(X)$
for the open fundamental complex  of $X$, and $T = \text{CFC}(X)$
for  the closed fundamental complex of $X$.
Our construction  of OFC$(X)$ comes with a basepoint  $b_0\in X$.
We shall describe these fundamental complexes below, because we shall
 construct an element in
$$ H^{lf}_{n+2} (\text{OFC}(X^n)\setminus \{b_0\}, \mathbb{\Omega}^{NPD}),$$
 associated to a degree one normal map $f: M^{n} \rightarrow X^{n}$,
  where $X^{n}$ is a generalized $n$-manifold,
 and
  $b: \nu_{M^n} \to \xi$ is
  an appropriate bundle map. 
  
  More precisely, we have to fix a degree one normal map
  $$
  \{f_0,b_0\}
  : M_0 \to X 
  $$
  and associate to 
  $
  \{f,b\} 
  $
  an element in
$$
 H^{lf}_{n+2} (\text{OFC}(X^n)\setminus \{b_0\}, \mathbb{\Omega}^{NPD}),
 $$
  which we shall denote
  $
  \{f,b\} -
  \{f_0,b_0\}.
  $    
   
By the above theorems  we have the following chain of morphisms: 
$$
H^{lf}_{n+2} (\text{OFC}(X^n)\setminus \{b_0\}, \mathbb{\Omega}^{NPD}) \xrightarrow[]{\cong}
H^{st}_{n+2} (\text{CFC}(X^n), X^n {\amalg}\{b_0\}, \mathbb{\Omega}^{NPD}) \xrightarrow[]{\cong}$$
$$\xrightarrow[]{\cong}H^{st}_{n+1} (\text{CFC}(X^n), X^n  {\amalg}\{b_0\}, \mathbb{L}^+) \rightarrow
H^{st}_{n} (X^n  {\amalg}\{b_0\}, \mathbb{L}^+),
$$ 
i.e. $
  \{f,b\} -
  \{f_0,b_0\}
  $
determines an element 
$$
[ f,b] - [f_0,b_0] \in H^{st}_{n} (X^n  {\amalg}\{b_0\}, \mathbb{L}^+).
$$ 

If $X^n$ is a topological $n$-manifold
which carries a simplicial structure,
 then the construction of the element 
$$[ f,b] - [f_0,b_0]  \in H^{st}_{n} (X^n, \mathbb{L}^+) \cong H_{n} (X^n, \mathbb{L}^+)$$
  follows
  from the splitting procedure: 
  the
  surgery problem
$$(f,b) = (M^n \xrightarrow{f} X^n, \ \ \nu_{M^n} \xrightarrow{b} \xi)$$
  can be split into
   adic surgery problems which define
     $\left[ f,b \right]$ (cf.
       Hegenbarth and Repov\v{s}~\cite{HeRe14}
       and
     Ranicki~\cite{Ranicki92}). 
This is due to transversality with respect to a dual cell
structure on $X^n$.  
 
 It is the purpose of this paper to present a construction, based on Theorems ~\ref{F:spectra},~ \ref{F:metricPair}, and~ \ref{F:compact},
 to obtain an element of $H^{st}_{n} (X^n, \mathbb{L}^+)$
  which
avoids this geometric splitting.
 We point out that algebraic
  splitting is also possible (cf.
Pedersen, Quinn, and Ranicki~\cite{PeQuRa03}) and it leads to an identification of $\mathbb{L}$-homology groups with controlled Wall groups. 

We conclude the introduction by describing the structure of our
 paper.  
In Section 2 we shall
recall preliminary material about nerves $N(\mathcal{U})$
and canonical maps $$\varphi: X^n \rightarrow N(\mathcal{U})$$
between the
underlying space $X^n$
and
the
 nerve
$N(\mathcal{U}).$ 

 Section 3 will be devoted to the construction of appropriate fundamental complexes of $X^n$.
 Section 4 is the core of the paper:
  for any ENR Poincar\'{e} duality space $X^n$
         we shall
         apply
         Theorems ~\ref{F:spectra},~ \ref{F:metricPair}, and~ \ref{F:compact}
         to construct  the $\mathbb{L}$-homology class $$\left[f,b\right] \in H^{st}_n (X^n, \mathbb{L}^+)$$
         for an arbitrary surgery problem
$$(f,b) = (M^n \xrightarrow{f} X^n, \ \ \nu_{M^n} \xrightarrow{b} \xi).$$
However, we shall see that this class depends on the canonical surgery problem
 (cf. Lemma~\ref{C:Lemma 4.3}).

In Section 5 we shall present some improvements and give an outlook. 
Finally, in Section 6 we shall discuss selected remaining related problems.         

For more background information on Poincar\'{e} complexes, surgery theory, and generalized manifolds
 we refer the reader to e.g. 
\cite{BFMW96, CaHeRe16, Ferry10, HaVo93, Qui87, Yamasaki87, Wal67, Wal99, Wei02}.         
\section{Coverings, nerves, and canonical maps}\label{S:Coverings}
Throughout the paper we shall consider compact metric spaces $X$. 
Our main interest will be
 {\it closed generalized n-manifolds} $X^n$, i.e.~$X^n$ is 
 a {\it Euclidean neighborhood retract} (ENR) and a $\mathbb{Z}$-{\it homology $n$-manifold}, i.e. 
 $$H_* (X^n, X^n \setminus \{x\}) \cong H_* (\mathbb{R}^{n}, \mathbb{R}^{n} \setminus \{0\}) \ \ \hbox{\rm{for every}} \ x \in X^n
  $$ (cf. e.g.
 Cavicchioli, Hegenbarth, and Repov\v{s}~\cite{CaHeRe16}).

Open coverings $\mathcal{U} = \{U\}_{j\in J}$
will 
always
be
 assumed to be locally finite. We shall denote
  the  simplicial complex of $\mathcal{U}$
  by $N (\mathcal{U})$. The vertex corresponding to
    $U_j \in \mathcal{U}$
    will be denoted by
    $\left< U_j \right>$, and if $$\underset{0\leq i \leq k}{\cap}{U_{j_i}} \neq \emptyset$$
    then the $k$-simplex determined by $U_{j_0}, \dots, U_{j_k}\in \mathcal{U}$ 
    will be denoted by $$\sigma = \left< U_{j_0}, \dots, U_{j_k} \right> \in N(\mathcal{U}). $$

We shall abbreviate and write $$\underset{0\leq i \leq k}{\cap}{U_{j_i}}  = \cap \sigma.$$
We
shall
 also write $N(\mathcal{U})$ for its topological realization.
 The space $N (\mathcal{U})$ can be given the Whitehead or the metric topology. 
 However, since we shall only  consider locally finite coverings, these two topologies
 are   identical (cf. e.g. Dugundji \cite[p.99]{Dugundji73}). 
\subsection{The map \texorpdfstring{$\varphi: X \rightarrow N(\mathcal{U})$}{y: X -> N(U)}}\label{SS:map}
~\\~
Let $${\hbox{\rm mesh}}(\mathcal{U}) = {\hbox{\rm sup}}\{ {\hbox{\rm diam}} (U) | U \in \mathcal{U}  \},$$ where ${\hbox{\rm diam}}(U)$ 
denotes
 the diameter of $U \subset X$.
We shall now describe the first one of our canonical maps.

A partition of unity $\{\varphi_j\}_{j\in J}$ subordinate to $\mathcal{U}$ gives rise to the map $\varphi: X \rightarrow N (\mathcal{U})$ defined by
$$\varphi(x) = \sum_{j} \varphi_j (x) \left< U_j \right>.$$ 
If $\{\overline{\varphi}_j\}_{j \in J}$ is another partition subordinate to $\mathcal{U}$, it defines
the map
 $\overline{\varphi}: X \rightarrow N(\mathcal{U})$. The homotopies $$\{ t{\overline{\varphi}_j} +  (1-t){{\varphi}_j}\mid 0\leq t \leq 1\}_{j\in J}$$ then define a homotopy between $\overline{\varphi}$ and $\varphi$, i.e.~up to homotopy, 
 the map $\varphi$ is unique. 
\subsection{Maps induced by refinements}\label{SS:refinements}
~\\~
Next, we shall consider refinements of coverings and induced maps.
Let $$\mathcal{U}' = \{ {U}'_{j'} \}_{j' \in J'}$$ be a refinement of $\mathcal{U}$, i.e.~there is a map $s: J' \rightarrow J$ such that $${U}'_{j'} \subset {U}_{s(j')}  \ \ \hbox{\rm{for every}} \  j'\in J'.$$

Let $\varphi' : X \rightarrow N (\mathcal{U}')$ 
be a map as defined in Section~\ref{SS:map} by the partition of unity $\{ {\varphi}'_{j'} \}_{j' \in J'}$. We want to complete the diagram
\begin{center}
\begin{tikzpicture}[node distance=2cm, auto]
  \node (LU) {};
  \node (X) [node distance=1cm, below of=LU] {$X$};
  \node (NU') [right of=LU] {$N (\mathcal{U}')$};
  \node (NU) [below of=NU'] {$N (\mathcal{U})$};
  \draw[->, font=\small] (X) to node {$\varphi'$} (NU');
  \draw[->, font=\small] (X) to node [swap] {$\varphi$} (NU);
  \draw[->, font=\small, dashed] (NU') to node {} (NU);
\end{tikzpicture}
\end{center}
by a map indicated by the dashed
 line, so that it is homotopy commutative
  even in the controlled way. 
  
  We can
  get such maps from e.g. Hu~\cite[Theorem~8.1, p.146]{Hu65}: There exists a refinement $\mathcal{V}$ of $\mathcal{U}$, such that for every refinement $\mathcal{U}'$ of $\mathcal{V}$ there is a simplicial map $N (\mathcal{U}') \rightarrow N (\mathcal{U})$ such that
\begin{center}
\begin{tikzpicture}[node distance=2cm, auto]
  \node (LU) {};
  \node (X) [node distance=1cm, below of=LU] {$X$};
    \node (NU') [right of=LU] {$N (\mathcal{U}')$};
  \node (NU) [below of=NU'] {$N (\mathcal{U})$};
  \draw[->, font=\small] (X) to node {$\varphi'$} (NU');
  \draw[->, font=\small] (X) to node [swap] {$\varphi$} (NU);
  \draw[->, font=\small] (NU') to node {} (NU);
\end{tikzpicture}
\end{center}
commutes up to a homotopy $h_t$ with 
$$\{ h_t (x)  | 0 \leq t \leq 1 \} \subset \overset{\circ}{st}\left< U \right> \ \  {\hbox{\rm {for some}}} \ \  U \in \mathcal{U}.$$
 Such maps are called {\it bridge maps} in
 Hu~\cite{Hu65}
 and
  {\it projections} in
 Milnor~\cite{Milnor95}.
\subsection{Maps from nerves to the space (dominations)}
~\\~
We now describe the construction of maps $N (\mathcal{U}) \rightarrow X$,
 using the construction presented by Ferry~\cite[Theorems 29.7 and 29.9, Section~29]{Ferry92}.
Given $$\sigma = \left< U_{j_0} \dots U_{j_k} \right> \in N (\mathcal{U}),$$
 we pick a point $x_\sigma \in \cap \sigma$
  and define a noncontinuous map $$\rho : N (\mathcal{U}) \xrightarrow[]{} X \ \mbox{ by} \  \rho (\sigma) = x_\sigma.$$ 
 
Let $W \subset \mathbb{R}^m$ be an appropriate regular neighbourhood of some embedding $X \subset \mathbb{R}^m$.
Then
the map
 $$N (\mathcal{U}) \xrightarrow[]{\rho} X \subset W$$ can be approximated by a continuous map $$\psi' : N (\mathcal{U}) \rightarrow W.$$ The composition with the retraction $$\pi : W \rightarrow X$$
 then
 gives
 the map
  $$\psi = \pi \circ \psi' : N (\mathcal{U}) \rightarrow X.$$ By sufficiently subdividing $N (\mathcal{U})$
   one can achieve
   that
    $${\hbox{\rm {dist}}} (\psi, \rho) < \delta \ \  {\hbox{\rm {for arbitrary small}}} \ \  \delta > 0.$$

For a given $\varepsilon > 0$, one
then
 finds coverings $\mathcal{U}$ with ${\hbox{\rm mesh}}(\mathcal{U})$ sufficiently small, so that 
 $${\hbox{\rm {dist}}}({Id}_X, \rho \circ \varphi) < \varepsilon,$$
  and therefore $${\hbox{\rm {dist}}}({Id}_X, \psi \circ \varphi) < \varepsilon + \delta.$$ 
  
  By
 invoking 
Ferry  \cite[Corollary 29.9]{Ferry92}, 
we can then conclude that 
 $${Id}_X \ \ {\hbox{\rm and}} \ \  \psi \circ \varphi \ \ {\hbox{\rm are}} \ \  \varepsilon'\hbox{\rm{-homotopic}}.$$
  Beginning with an $\varepsilon' > 0$, one
 then
  finds coverings $\mathcal{U}$ of $X$ such that  
  $$\psi \circ \varphi \ \ {\hbox{\rm is}} \ \  \varepsilon'\hbox{\rm{-homotopic to}} \ \ {Id}_X .$$
  
This is well-known (cf. e.g.
Hu~\cite[Theorem~6.1, p.138]{Hu65}), but we
shall need some of the details from above in the sequel.  
\begin{theorem}[Hu~\cite{Hu65}]
Let $X$ be an absolute neighborhood retract (ANR). Then the following properties hold:
\begin{enumerate}
	\item[(i)] Given an open covering $\mathcal{U}$ of $X$, there exist maps 
	$$\varphi: X \rightarrow N(\mathcal{U}) \ \mbox{and} \  \psi : N(\mathcal{U}) \rightarrow X.$$
	\item[(ii)] Given $\varepsilon > 0$, there exists an open  covering
	 $\mathcal{U}_{\varepsilon}$ of $X$ such that the composite map
	$$
	X \xrightarrow{\varphi} N(\mathcal{U}_{\varepsilon}) \xrightarrow{\psi} X
	$$
	 is $\varepsilon$-homotopic to ${Id}_X$.
	\item[(iii)] Given $\mathcal{U}_{\varepsilon}$ as in (ii), there exist
	a refinement $\mathcal{U}'_{\varepsilon}$ and a
	map $N(\mathcal{U}'_{\varepsilon}) \rightarrow N(\mathcal{U}_{\varepsilon})$ such that 	
	the diagram and its  subdiagrams 
	\begin{center}
	\begin{tikzpicture}[node distance=2cm, auto]
	  \node (LU) {};
	  \node (X) [node distance=1cm, below of=LU] {$X$};
	  \node (NU') [right of=LU] {$N (\mathcal{U}'_{\varepsilon})$};
	  \node (NU) [below of=NU'] {$N (\mathcal{U}_{\varepsilon})$};
	  \node (RU) [right of=NU'] {};
	  \node (X2) [node distance=1cm, below of=RU] {$X$};
	  \draw[->, font=\small] (X) to node {$\varphi'$} (NU');
	  \draw[->, font=\small] (X) to node [swap] {$\varphi$} (NU);
	  \draw[->, font=\small] (NU') to node {} (NU);
	  \draw[->, font=\small] (NU') to node {$\psi'$} (X2);
	  \draw[->, font=\small] (NU) to node [swap] {$\psi$} (X2);
	\end{tikzpicture}
	\end{center}	
	are commutative up to $\varepsilon$-homotopy.
\end{enumerate}
\end{theorem}
\section{Fundamental complexes}\label{SS:3.1}
Let $X$ be a compact metric space. As explained in Section~\ref{S:Coverings}, we can choose a covering $\mathcal{U}$ of $X$ such that the composite map
$$X \xrightarrow{\varphi} N(\mathcal{U}) \xrightarrow{\psi} X$$
 is an $\varepsilon$-equivalence for a given $\varepsilon > 0$. Then we 
can
 choose a refinement $\mathcal{U}'$ of $\mathcal{U}$ such that  the composite map
 $$X \xrightarrow{\varphi'} N(\mathcal{U}') \xrightarrow{\psi'} X$$
  is an $\varepsilon'$-equivalence for a given $\varepsilon' < \varepsilon$, etc. 
 
 In this way we
 can get
  a sequence of coverings 
  $\{ \mathcal{U}_1, \mathcal{U}_2, \dots \}$
   such that
    $\mathcal{U}_{j+1}$
    refines
    $\mathcal{U}_{j}$ for every
    $j\in \mathbb{N}$, and  the composite map
    $$X \xrightarrow{\varphi_j} N(\mathcal{U}_j) \xrightarrow{\psi_j} X$$ is an $\varepsilon_j$-equivalence with $\varepsilon_j \rightarrow 0$ for $j \rightarrow \infty$.     
Moreover, we have simplicial maps $$s_j : N(\mathcal{U}_{j+1}) \rightarrow N(\mathcal{U}_{j})$$ so that the diagram and its subdiagrams
\begin{center}
\begin{tikzpicture}[node distance=2cm, auto]
  \node (LU) {};
  \node (X) [node distance=1cm, below of=LU] {$X$};
  \node (NU') [right of=LU] {$N (\mathcal{U}_{j+1})$};
  \node (NU) [below of=NU'] {$N (\mathcal{U}_{j})$};
  \node (RU) [right of=NU'] {};
  \node (X2) [node distance=1cm, below of=RU] {$X$};
  \draw[->, font=\small] (X) to node {$\varphi_{j+1}$} (NU');
  \draw[->, font=\small] (X) to node [swap] {$\varphi_j$} (NU);
  \draw[->, font=\small] (NU') to node {$s_j$} (NU);
  \draw[->, font=\small] (NU') to node {$\psi_{j+1}$} (X2);
  \draw[->, font=\small] (NU) to node [swap] {$\psi_j$} (X2);
\end{tikzpicture}
\end{center}
commute up to homotopy.

We add to this sequence the trivial covering $\mathcal{U}_0 = \{ X \}$ with 
$$s_0 : N(\mathcal{U}_1) \rightarrow N(\mathcal{U}_{0}) = \left< X \right>$$
 the constant map. The union of the mapping cylinders of the simplicial  maps
  $\{ s_j \}_{j\in\{0,1,\dots\}}$, denoted as
$$F = \underset{j \geq 0}{\cup} N (\mathcal{U}_{j+1}) \times I \underset{s_j}{\cup}N (\mathcal{U}_{j})$$
 is an open fundamental complex of $X$.

Let 
$$F_l = \underset{l \geq j \geq 0}{\cup} N (\mathcal{U}_{j+1}) \times I \underset{s_j}{\cup}N (\mathcal{U}_{j}),$$
 i.e. $F_l \subset F_{l+1}$ is a deformation retract
 and
  $$r_l : F_{l+1} \rightarrow F_l$$
  is  the obvious retraction. Then $$CF = \underset{l}{\varprojlim} F_l$$ is a closed fundamental complex.  
  
Both complexes are contractible, $F \subset CF$, and 
$$CF \setminus F =  \underset{s_j}{\varprojlim} N (\mathcal{U}_{j+1}).$$ 
Identifying $N (\mathcal{U}_{j})$ with the mapping cylinder 
$$N (\mathcal{U}_{j+1}) \times I \underset{s_j}{\cup} N (\mathcal{U}_{j}),$$
 we can
 form $\underset{j}{\cap} N (\mathcal{U}_{j})$ and complete $F$
 by it, i.e. 
 $$\underset{j}{\cap} N (\mathcal{U}_{j}) = \underset{s_j}{\varprojlim} N (\mathcal{U}_{j+1}).$$  
\begin{theorem}\label{L:maps}
The maps $$\psi_j : N (\mathcal{U}_{j}) \rightarrow X$$ fit together to form  the
 map $$\psi: \underset{j}{\varprojlim} N (\mathcal{U}_{j}) \rightarrow X.$$
\end{theorem}
\begin{proof}
Let 
$$h: N (\mathcal{U}_{j+1}) \times I \rightarrow X$$ 
be a homotopy between $\psi_{j+1}$ and $\psi_j \circ s_j$. It induces a map 
$$\Lambda_j : N (\mathcal{U}_{j+1}) \times I \underset{s_j}{\cup} N (\mathcal{U}_{j}) \rightarrow X$$
 which restricts to $$\psi_{j+1} \ \mbox{ on} \ N (\mathcal{U}_{j+1}) \times \{0\} \ \  \mbox{ and}  \ \ \psi_j \  \mbox{on} \ N (\mathcal{U}_{j}),$$ hence it can be glued to give maps
$$F^\circ_l = \underset{l \geq j \geq 1}{\cup} N (\mathcal{U}_{j+1}) \times I \rightarrow X.$$
Since the diagram
\begin{center}
\begin{tikzpicture}[node distance=2.5cm, auto]
  \node (N1) {$N (\mathcal{U}_{j+1}) \times I \underset{s_j}{\cup} N (\mathcal{U}_{j})$};
  \node (M) [right of=N1] {};
  \node (N2) [right of=M] {$N (\mathcal{U}_{j})  = N (\mathcal{U}_{j}) \times \{0\} $};
  \node (X) [node distance=1.5cm, below of=M] {$X$};
  \draw[->, font=\small] (N1) to node {$r_{j}$} (N2);
  \draw[->, font=\small] (N1) to node {} (X);
  \draw[->, font=\small] (N2) to node {} (X);
\end{tikzpicture}
\end{center}
commutes, it induces a map $$\underset{l \geq 1}{\varprojlim} F^\circ_l \rightarrow X$$
whose restriction to $\underset{l \geq 1}{\varprojlim} N (\mathcal{U}_{l})$ 
then gives the map $\psi$.
\end{proof}
\section{Construction of $\mathbb{L}$-homology classes}\label{SS:3.2}
In this section,
 $X^n$ will
denote an oriented generalized $n$-manifold, $n\ge 5$, with
a fundamental class $$\left[ X \right] \in H_n (X^n, \mathbb{Z}).$$ 
Then $X^n$ has a Spivak normal fibration $\nu_{X^n}$ (cf. 
Quinn~\cite[Example 2.3]{Qui83}). 
Moreover, 
$\nu_{X^n}$
has topological reductions (cf. 
Ferry and Pedersen~\cite[Theorem 16.6]{Fe-Ped}).
We shall
consider a sequence of coverings $\{\mathcal{U}_j\}_{j\in\{0,1,\dots\}}$ as described in Section~\ref{SS:3.1}.
\begin{theorem}\label{L:mapGamma}
There is a map
$$\Gamma_j : X^n \times I \rightarrow N (\mathcal{U}_{j+1}) \times I \underset{s_j}{\cup} N (\mathcal{U}_{j}) $$
such that $\Gamma_j$ restricts to 
$$\varphi_{j+1} : X^n \times \{0\} \rightarrow N (\mathcal{U}_{j+1}) \times \{0\}$$
 and 
  $$\varphi_j : X^n \times \{1\} \rightarrow N (\mathcal{U}_{j}).$$
\end{theorem}
\begin{proof}
We consider the composite map 
$$\overline{\varphi}_{j+1} : X^n \times I \xrightarrow{\varphi_{j+1} \times Id} N (\mathcal{U}_{j+1})  \times I \rightarrow N (\mathcal{U}_{j+1}) \times I \underset{s_j}{\cup} N (\mathcal{U}_{j}).$$
It restricts to 
$$\varphi_{j+1}: X^n \times \{0\} \rightarrow N (\mathcal{U}_{j+1}) \times I \underset{s_j}{\cup} N (\mathcal{U}_{j})$$ and 
$$s_j \circ \varphi_{j+1} : X^n \times \{1\} \rightarrow N (\mathcal{U}_{j+1})  \times I \underset{s_j}{\cup} N (\mathcal{U}_{j}).$$

However,
 $$s_j \circ \varphi_{j+1} \simeq \varphi_j \ \mbox{ via  the
   homotopy} \  g: X^n \times I \rightarrow N (\mathcal{U}_{j}).$$ 
 Composing 
  $\overline{\varphi}_{j+1}$ and $g$~in the obvious way,
$$X^n \times I \cup X^n \times \left[ 1,2\right] \xrightarrow{\overline{\varphi}_{j+1} \cup g} N (\mathcal{U}_{j+1}) \times I \underset{s_j}{\cup} N (\mathcal{U}_{j}),$$
one gets the required map 
$$\Gamma_j : X^n \times I \simeq X^n \times I \cup X^n \times \left[ 1,2\right] \rightarrow N (\mathcal{U}_{j+1}) \times I \underset{s_j}{\cup} N (\mathcal{U}_{j}).$$
\end{proof}
Let us denote 
$$F_0 = \underset{j \geq 1}{\cup} N (\mathcal{U}_{j+1}) \times I \underset{s_j}{\cup} N (\mathcal{U}_{j}),$$
 i.e. $F_0 \sim F \setminus \{b_0\}$, where 
 $$b_0 \in N (\mathcal{U}_{1}) \times I \underset{s_0}{\cup} N (\mathcal{U}_{0})$$ is the base point
 of
  $N (\mathcal{U}_{0}).$
  Then we 
  get the following corollary.
\begin{corollary}\label{C:mapsLambda}
The maps $\Lambda_j$ and $\Gamma_j$
in
Theorems~\ref{L:maps} and~\ref{L:mapGamma} fit together to give maps
$$X^n \times \mathbb{R}_+ \xrightarrow{\Gamma} F_0 \xrightarrow{\Lambda} X^n$$
such that $\Lambda \circ \Gamma$ restricts to $$\psi_1 \circ \varphi_1 : X^n \times \{0\} \rightarrow X^n.$$
\end{corollary}
We can now construct a normal space with underlying space $F_0$ as follows:
Let 
$\xi$ be a topological reduction of $\nu_{X^n}$ and set
$\eta = \Lambda^* (\xi)$. Since 
$$\Lambda \circ \Gamma \sim \psi_1 \circ \varphi_1 \sim Id,$$
 we get $$\Gamma^* (\eta) \cong \xi \times \mathbb{R}_+.$$ Then 
 $$\beta: S^m \times \mathbb{R}_+ \xrightarrow{\alpha \times Id} T (\xi) \times \mathbb{R}_+ \cong T (\Gamma^* (\eta)) \rightarrow T (\eta)$$
  defines the structure map of the bundle $\eta$ over $F_0$. Here, $T (.)$ denotes the Thom space and 
  $$\alpha : S^m \rightarrow T (\xi) \sim T (\nu_{X^n})$$
   is the structure map of $(X^n, \nu_{X^n})$, where we assume
  that
   $X^n \subset S^m$. 
   Therefore $\xi$ is an $\mathbb{R}^{m-n}$-bundle over $X^n$.

Let
$$(f,b) = (f: M^n \rightarrow X^n, \ \ b: \nu_{M^n} \rightarrow \xi)$$
 be a surgery problem. 
 It
 defines a normal map 
 $$(F, B) = (M^n \times \mathbb{R}_+ \xrightarrow{f \times Id} X^n \times \mathbb{R}_+ \xrightarrow{\Gamma} F_0, \ \  \nu_{M^n} \times \mathbb{R}_+ \xrightarrow{b \times Id} \xi \times \mathbb{R}_+ \xrightarrow{\widetilde{\Gamma}} \eta),$$
  where $\widetilde{\Gamma}$ is the obvious bundle map covering $\Gamma$.
  
The mapping cylinder $M (F,B)$ is well-known to be a normal space with the boundary equal to 
$$
\partial M(F,B)=M^n \times \mathbb{R}_{+}\amalg F_0.
$$
We shall only consider the restriction of $(F,B)$ to 
$$M^n \times (0, \infty) \rightarrow F \setminus N (\mathcal{U}_{1}) \times I \underset{s_0}{\cup} N (\mathcal{U}_{0}),$$
 and also denote it by $(F,B)$. Since 
 $$F \setminus N (\mathcal{U}_{1}) \times I \underset{s_0}{\cup} N (\mathcal{U}_{0})$$
  is a locally finite complex,
   normal transversality is used to decompose $M(F,B)$ into adic normal complexes.

If $\xi'$ is another topological reduction of $\nu_{X^n}$
then the same construction gives $M(F',B')$.
One now glues
$$M(F,B)    \cup  -M(F',B') \ \ \hbox{\rm{along}} \ \  F_0$$ to obtain an element
$$\{f',b'\}-\{f,b\} \in   H^{lf}_{n+2} (F_0, \mathbb{\Omega}^{NPD}). $$ 
Here, $-M(F,B)$
indicates that the orientation
on $M^n$ is reversed.

By Theorem~\ref{F:spectra}, this is isomorphic to 
$$H^{lf}_{n+1} (F_0,\mathbb{L}^{+})$$
 which in turn,
  is by Theorem~\ref{F:metricPair}, isomorphic to 
 $$
 H^{st}_{n+1} (CF, \underset{j}{\varprojlim} N (\mathcal{U}_{j})  {\amalg}  N (\mathcal{U}_{1}) \times I \underset{s_0}{\cup} N (\mathcal{U}_{0}), \mathbb{L}^+).
 $$
  Under the homology boundary morphism it maps to an element in
   $$H^{st}_{n} ( \underset{j}{\varprojlim} N (\mathcal{U}_{j}), \mathbb{L}^+).$$
Finally,
$$\psi_* : H^{st}_{n} ( \underset{j}{\varprojlim}N (\mathcal{U}_{j}), \mathbb{L}^+) \rightarrow H^{st}_{n} ( X, \mathbb{L}^+)$$
 gives the desired element 
 $[f',b']-[f,b]$.
\begin{remark}\label{C:remark} We thank the referee for pointing out an error here in our previous version (we have claimed that $M(F,B)$ defines an element already  in  $H^{lf}_{n+2} (F_0,\mathbb{\Omega}^{NPD})$). 
\end{remark}
The element
 $\{f',b'\}-\{f,b\}$
 is represented by a compatible family of adic objects
  $$(\{f',b'\}-\{f,b\})_{\sigma}$$
  belonging to the semisimplicially defined
  spectrum 
$\mathbb{\Omega}^{NPD}$,
where $\sigma$ is a simplex in $F_0$.
Since $\sigma$ belongs to some 
$$N (\mathcal{U}_{l+1}) \times I \underset{s_l}{\cup} N (\mathcal{U}_{l})$$
one can break
 $\{f',b'\}-\{f,b\}$
into pieces
 $$
 \{f',b'\}_{l}-\{f,b\}_{l}\in
H^{lf}_{n+2} (N (\mathcal{U}_{l+1}) \times I \underset{s_l}{\cup} N (\mathcal{U}_{l})   ,\mathbb{\Omega}^{NPD})
= $$
$$
= H_{n+2} (N (\mathcal{U}_{l+1}) \times I \underset{s_l}{\cup} N (\mathcal{U}_{l})   ,\mathbb{\Omega}^{NPD}).
 $$
 
 We shall return to this splitting later on. A detailed construction of the adic elements
 $$(\{f',b'\}-\{f,b\})_{\sigma}$$ which 
 works also in our case 
 is given in
 K\" uhl, Macko, and Mole~\cite[Construction 11.3, p.236]{KMM13}.
 
 Before stating the main result of this section recall the following well-known fact
 (cf. Ferry and Pedersen \cite[Theorem 16.6]{Fe-Ped}). 
 \begin{lemma}\label{C:Lemma 4.3}
 The canonical topological reduction $\xi_0$ of the generalized manifold $X$ defines, up to a normal cobordism, a unique surgery problem
 $$
 (f_0, b_0): M_0 \to X,
$$
 called the canonical surgery problem.
 \end{lemma}
 \begin{proof}
 Since $X$ is a compact metric space,
 it is homotopy equivalent to a finite complex $K$, hence $K$ is a $PD_n$-complex  (cf. West~\cite{West77}).
 There is a fiber homotopy equivalence
 $ \nu_K \sim  \nu_X$
 covering the homotopy equivalence 
 $K \sim  X$.
 The latter induces a reduction $\xi_0$ on $K$ and a structure map
   $$S^m \to T(\xi_0).$$
   
 Transversality applies here to define a surgery problem (the Pontryagin-Thom construction):
 $$ 
 (f_0, b_0): M_0 \to K\sim X.
 $$
 If $K'$ is another finite complex homotopy equivalent to $X$, it can be easily proved that the resulting surgery problem is normally cobordant to
  $(f_0, b_0)$.
 \end{proof} 
 In summary, we have obtained the following result.
 \begin{theorem}\label{C:Proposition 4.5}
 Let $X$ be an oriented generalized $n$-manifold, $n \ge 5$, 
 with the canonical reduction $\xi_0$ of $\nu_X$
 whose associated canonical surgery problem is
 $$(f_0, b_0): M_0 \to X.$$
 Then the procedure explained
 after Corollary~\ref{C:mapsLambda}
 yields
  for any degree one normal map
 $$(f, b): M \to X,$$
 a well-defined element
 $$ [f, b]-[ f_0, b_0]
 \in
 H^{st}_{n} (X,\mathbb{L}^{+}).
 $$ 
 \end{theorem}
\section{Improvements and the outlook}\label{SS:Ch5}
As in the previous chapter, let $X$ be a generalized manifold and $\mathcal{N}(X)$ the set of all normal bordism classes of degree one normal maps 
 $(f, b): M \to X$.
 Theorem~\ref{C:Proposition 4.5}
 can be improved to the following theorem. 
 \begin{theorem}\label{C:Theorem 5.1}
 The association
 $$(f,b) \rightsquigarrow   [f, b]-[ f_0, b_0]$$
  in Theorem~\ref{C:Proposition 4.5}
  defines a map
  $$ 
  t:\mathcal{N}(X) \to 
 H^{st}_{n} (X,\mathbb{L}^{+}).
  $$
 \end{theorem}
 \begin{proof}
 We first 
 have to show 
 that the construction does not depend on the choice of the normal 
 bordism 
 class of $(f,b)$
 and second, 
 that it does not depend on the choice of the sequence
 $\{\mathcal U_j\}$
 described in Chapter~\ref{SS:3.1}
 either. 
 \begin{lemma}\label{C:Lemma 5.2}
 Fix the sequence of coverings
 $\{\mathcal U_j\}$ of $X$
 described in 
 Chapter~\ref{SS:3.1}.
 Suppose that
  $$(f,b):M\to X$$
 is normally bordant to 
 $$
 (f', b'): M' \to X.
 $$
 Then
 $$
 \{f',b'\}-\{f_0,b_0\}=
   \{f,b\}-\{f_0,b_0\}\in
   H^{lf}_{n+2} (F_0,\mathbb{\Omega}^{NPD}).
 $$
 \end{lemma} 
 \begin{proof}
 Let 
 $$
 (g,c):W\to X\times I
 $$
 be a normal cobordism between $(f,b)$ and $(f',b')$.
 Consider also the product normal cobordism
 $$
 (g_0,c_0):M_0\times I \to X\times I.
 $$
 
 The mapping cylinders of the obvious normal maps
 $$
 (G,C):W \times \mathbb R_{+} \to X \times I \times \mathbb R_{+} \to F_0 \times I
 $$
and
 $$
  (G_0,C_0):M_0 \times I  \times \mathbb R_{+} \to X \times I \times \mathbb R_{+} \to F_0 \times I
 $$
 can be glued along
 $F_0 \times I$
 to give a normal cobordism between
 $$M(F,B)    \cup  -M(F_0,B_0)$$ 
 and
  $$M(F',B')    \cup  -M(F_0,B_0),$$
  implying the claim.
(For definitions of $M(F,B), -M(F_{0},B_{0}),$ and $M(F',B')$ see the previous section.) 
 \end{proof} 
 For the second step of the proof of Theorem \ref{C:Theorem 5.1}, we let
 $
 \{\mathcal U_j\}, 
 \{\mathcal U'_j\}
 $
 be two sequences defining
 $$
 F_0=\underset{j \ge 1}{\cup} N (\mathcal{U}_{j+1}) \times I \underset{s_j}{\cup} N (\mathcal{U}_{j})
 $$
 and
 $$
 F'_0= \underset{j \ge 1}{\cup} N (\mathcal{U}'_{j+1}) \times I \underset{s'_j}{\cup} N (\mathcal{U}'_{j}).
 $$
 Let 
 $$
 \mathcal U''_j=\{U \cap U' \mid \  U \in \mathcal U_j,  U' \in \mathcal U'_j \}.
 $$
 Observe that
 $$\hbox{\rm{mesh}} 
 (\mathcal U''_j)
  \le 
 \hbox{\rm{min}}
 \{
 \hbox{\rm{mesh}} (\mathcal U_j), \
 \hbox{\rm{mesh}} (\mathcal U'_j)
 \},
 $$
 where
 $
 \{\mathcal U''_j\}
 $
 is a sequence as described in Chapter~\ref{SS:3.1},
 defining $F''_0$ and the open (resp. closed) fundamental complex $F''$ (resp. $CF''$).
  
 Let 
 $
 (f,b):M\to X
 $
 be given.
 Our strategy
 will be to compare the elements
 $$
 \{f,b\}-\{f_0,b_0\}
 \in
   H^{lf}_{n+2} (F_0,\mathbb{\Omega}^{NPD})
 $$
 and
 $$
 \{f,b\}'-\{f_0,b_0\}'
 \in
   H^{lf}_{n+2} (F'_0,\mathbb{\Omega}^{NPD}),
 $$
 with
 $$
 \{f,b\}''-\{f_0,b_0\}''
 \in
   H^{lf}_{n+2} (F''_0,\mathbb{\Omega}^{NPD}).
 $$
 
 Since
 $
 \mathcal U''_j
 $
 refines 
 $
 \mathcal U_j,
 $
 there are maps
 $$
 p_j:N(\mathcal U''_j)
  \to
  N(  \mathcal U_j)
 $$
 such that the diagram 
\begin{center}
\begin{tikzpicture}[node distance=2.5cm, auto]
  \node (Xl) {$N(\mathcal U''_{j+1})$};
  \node (Fl) [right of=Xl] {$N(\mathcal U''_{j})$};
  \node (Xl1) [below of=Xl] {$N(\mathcal U_{j+1})$};
  \node (Fl1) [below of=Fl] {$N(\mathcal U_{j})$};
  \draw[->, font=\small] (Xl) to node [midway, above]{$s''_j$} (Fl);
  \draw[->, font=\small] (Xl) to node {$p_{j+1}$} (Xl1);
  \draw[->, font=\small] (Fl) to node {$p_j$} (Fl1);
  \draw[->, font=\small] (Xl1) to node [midway, below]{$s_j$} (Fl1);
\end{tikzpicture}
\end{center}
 commutes up to homotopy.
 
  The mapping cylinder  construction now applies 
 to obtain
  maps
 $$
 q_j: 
 N (\mathcal{U}''_{j+1}) \times I \underset{s''_j}{\cup} N (\mathcal{U}''_{j})
 \to
 N (\mathcal{U}_{j+1}) \times I \underset{s_j}{\cup} N (\mathcal{U}_{j})
 $$
 which restrict to 
 $p_{j+1}$ 
 (resp. $p_j$)
 on the boundary.
 Therefore they can be pieced together to yield a map
 $$
 q=\cup q_j:F''_0 \to F_0.
 $$
 
 The completion of this process then gives the map which we shall also denote by 
 $q$,
 $$
 q:(CF'', F''_0, \underset{j}{\varprojlim} N(\mathcal U''_j))
  \to 
  (CF, F_0, \underset{j}{\varprojlim} N(\mathcal U_j)).
 $$ 
 We shall also need the following lemma.
 \begin{lemma}\label{C:Lemma 5.3}
 Under the map
$$
q_{*} :
   H^{lf}_{n+2} (F''_0,\mathbb{\Omega}^{NPD})
\to
   H^{lf}_{n+2} (F_0,\mathbb{\Omega}^{NPD}),
 $$
 the element
 $$
 \{f,b\}''-\{f_0,b_0\}''
 $$
 maps to the element
 $$
 \{f,b\}-\{f_0,b_0\}.
 $$
 \end{lemma} 
 \begin{proof}
 To prove the lemma, we ``break up''
 $$
 \{f,b\}''-\{f_0,b_0\}''
 \ \mbox{(resp. } 
 \{f,b\}-\{f_0,b_0\})$$
   into pieces 
$$ \{f,b\}''_j-\{f_0,b_0\}''_j
 \ \mbox{ (resp. } 
 \{f,b\}_j-\{f_0,b_0\}_j
 )$$ 
 and we show that they correspond under
 $q_{*}$.
 
 To this end, consider the normal map
 $$
 (F_j, B_j):
 M \times [j,j+1]
 \xrightarrow{(f,b) \times Id}
 X \times [j,j+1] 
 \xrightarrow{\Gamma_j}
 N (\mathcal{U}_{j+1}) \times I \underset{s_j}{\cup} N (\mathcal{U}_{j}),
 $$
 where $B_j$
 is the obvious bundle map with the target in 
 $\Lambda^{*}_j(\xi).$
 As above, here $b:\nu_M \to \xi$
 is the bundle map of
 $(f,b):M\to X.$
 Observe that
 $$
 \restr{ \Lambda^{*}(\xi)}{N (\mathcal{U}_{j+1}) \times I \underset{s_j}{\cup} N (\mathcal{U}_{j})}
 \cong \Lambda^{*}_j(\xi).
 $$
 The mapping cylinder $M(F_j, B_j)$
 is then a normal complex with boundary. 
 We do the same for
 $$
 (f_0,b_0):M_0\to X
 $$
 and obtain
 $(F^{\circ}_j, B^{\circ}_j).$
 
 Then
 $$
 M(F_j, B_j)
 \cup
 -M(F^{\circ}_j, B^{\circ}_j)
 $$
 defines
 an element
 $$
 \{f,b\}_j-\{f_0,b_0\}_j
 \in
 H_{n+2}(N (\mathcal{U}_{j+1}) \times I \underset{s_j}{\cup} N (\mathcal{U}_{j}), \mathbb{\Omega}^{NPD}) .
 $$
 
 The inclusions
 $$
 N (\mathcal{U}_{j+1}) \times I \underset{s_j}{\cup} N (\mathcal{U}_{j}) \to F_0
 $$
 represent
 $
 \{f,b\}-\{f_0,b_0\}
 $
 as an infinite (locally finite) sum 
 $$
\underset{j} \Sigma (\{f,b\}_j-\{f_0,b_0\}_j).
 $$
 
 The same process yields a representation for 
 $
 \{f,b\}''-\{f_0,b_0\}''
 $
 as an infinite (locally finite) sum
 $$
 \underset{j} \Sigma (\{f,b\}''_j-\{f_0,b_0\}''_j). 
 $$
 
We now consider the following (homotopy) commutative diagram
\begin{center}
\begin{tikzpicture}[node distance=4.0cm, auto]
  \node (A) {$M\times I$};
  \node (B) [right of=A] {$X\times I$};
   \node (C) [right of=B] {$N (\mathcal{U}''_{j+1}) \times I \underset{s''_j}{\cup} N (\mathcal{U}''_{j}) $};
    \node (D) [right of=C] {$X$};
  \node (E) [below of=A] {$M\times I$};
  \node (F) [below of=B] {$X\times I$};
   \node (G) [below of=C] {$N (\mathcal{U}_{j+1}) \times I \underset{s_j}{\cup} N (\mathcal{U}_{j}) $};
  \node (H) [below of=D] {$X$};
  \draw[->, font=\small] (A) to node {$f\times Id$} (B);
  \draw[->, font=\small] (B) to node {$\Gamma''_j$} (C);
  \draw[->, font=\small] (C) to node {$\Lambda''_j$} (D);
  \draw[->, font=\small] (A) to node {$=$} (E);
  \draw[->, font=\small] (B) to node {$=$} (F);
  \draw[->, font=\small] (C) to node {$q_j$} (G);
  \draw[->, font=\small] (D) to node {$=$} (H);
  \draw[->, font=\small] (E) to node {$f\times Id$} (F);  
  \draw[->, font=\small] (F) to node {$\Gamma_j$} (G);
  \draw[->, font=\small] (G) to node {$\Lambda_j$} (H);
\end{tikzpicture}
\end{center}
First, we deduce from this diagram that
$$
q_{j}^{*} \Lambda_j^{*} (\xi)
\cong
 {{\Lambda}''_{j}}^{*} (\xi),
$$
hence $q_j$ can be covered  by a bundle map
$$
 \overline{q}_{j}: {{\Lambda}''_{j}}^{*} (\xi) \to {\Lambda}_{j}^{*}(\xi).
 $$
Next, we have the following
$$
 {{\Gamma}''_{j}}^{*}({{\Lambda}''_{j}}^{*} (\xi))  \cong  \xi \times I 
 \ \ \     \mbox{and}     \ \ \
   {\Gamma}_{j}^{*}({\Lambda}_{j}^{*} (\xi))  \cong  \xi \times I.
$$
From this we obtain the bundle maps
$$
b'': \nu_M \times I \xrightarrow{b \times Id} \xi \times I \cong  {{\Gamma}''_{j}}^{*}({{\Lambda}''_{j}}^{*} (\xi)) 
$$
and
$$
b: \nu_M \times I \xrightarrow{b \times Id} \xi \times I \cong  {\Gamma}_{j}^{*}({\Lambda}_{j}^{*} (\xi)) 
$$
together with the following (homotopy) commutative diagram
\begin{center}
\begin{tikzpicture}[node distance=4.0cm, auto]
  \node (A) {};
  \node (B) [right of=A] {};
   \node (C) [right of=B] {$T({{\Gamma}''_{j}}^{*}({{\Lambda}''_{j}}^{*}(\xi)))$}; 
  \node (E) [below of=A] {$S^m \times I$};
  \node (F) [below of=B] {$T(\nu_M)\times I$};
   \node (G) [below of=C] {};    
  \node (X) [below of=E] {};
  \node (Y) [below of=F] {};
   \node (Z) [below of=G] {$T({\Gamma}_{j}^{*}({\Lambda}_{j}^{*}(\xi)))$}; 
  \draw[->, font=\small] (E) to node {$\beta''$} (C);  
  \draw[->, font=\small] (E) to node {$\alpha \times Id$} (F);  
  \draw[->, font=\small] (E) to node {$\beta$} (Z);    
  \draw[->, font=\small] (F) to node {} (C);     
  \draw[->, font=\small] (F) to node {} (Z);   
 \draw[->, font=\small] (C) to node {$T(\overline{q}_j)$} (Z); 
\end{tikzpicture}
\end{center}
Now, 
$$
(N(\mathcal{U}''_{j+1}) \times I \underset{s''_j}{\cup} N (\mathcal{U}''_{j}), 
{{\Gamma}''_{j}}^{*}({{\Lambda}''_{j}}^{*}(\xi)),
\beta'') 
$$
and
$$
(N (\mathcal{U}_{j+1}) \times I \underset{s_j}{\cup} N (\mathcal{U}_{j}), 
{\Gamma}_{j}^{*}({\Lambda}_{j}^{*}(\xi)),
\beta)
$$
determine
$M(F''_j, B''_j)$ and $M(F_j,B_j),$ respectively.
Therefore we can conclude that 
$$
(*) \ \ \ \ \ \ 
q_{j^{*}}:\Omega_{n+2}^{NPD}(N(\mathcal{U}''_{j+1}) \times I \underset{s''_j}{\cup} N (\mathcal{U}''_{j}))
\to
\Omega^{NPD}_{n+2}(N(\mathcal{U}_{j+1}) \times I \underset{s_j}{\cup} N (\mathcal{U}_{j}))
$$
maps
$M(F''_j, B''_j)$ to $M(F_j,B_j)$.

The same holds for
$M(F''^{\circ}_j, B''^{\circ}_j)$ and $M(F^{\circ}_j, B^{\circ}_j),$
if we take
$(f_{0}, b_{0}): M_0\to X$
instead of 
$(f, b): M\to X.$

Since the differences have manifold boundaries, we get
$$
M(F''_j, B''_j) - M(F''^{\circ}_j, B''^{\circ}_j)
\in
\Omega_{n+2}^{NTOP}(N(\mathcal{U}''_{j+1}) \times I \underset{s''_j}{\cup} N (\mathcal{U}''_{j})
\cong
$$
$$
H_{n+2}(N(\mathcal{U}''_{j+1}) \times I \underset{s''_j}{\cup} N (\mathcal{U}''_{j}), \mathbb{\Omega}^{NTOP}),
$$
and similarly,
$$
M(F_j, B_j) - M(F^{\circ}_j, B^{\circ}_j)
\in
\Omega_{n+2}^{NTOP}(N(\mathcal{U}_{j+1}) \times I \underset{s_j}{\cup} N (\mathcal{U}_{j}))
\cong
$$
$$
H_{n+2}(N(\mathcal{U}_{j+1}) \times I \underset{s_j}{\cup} N (\mathcal{U}_{j}), \mathbb{\Omega}^{NTOP})
$$
(cf. Remark~\ref{C:remark2} below). 

The canonical map of spectra
$\mathbb{\Omega}^{NTOP}
\to 
\mathbb{\Omega}^{NPD}
$
maps these elements to
$$
\{f,b\}''_j-\{f_0,b_0\}''_j
\ \
\mbox{and} 
\ \
\{f,b\}_j-\{f_0,b_0\}_j
,$$
respectively.

The property (*) above now implies that
$$
q_{j^{*}}:
H_{n+2}(N(\mathcal{U}''_{j+1}) \times I \underset{s''_j}{\cup} N (\mathcal{U}''_{j}), \mathbb{\Omega}^{NPD})
\to
$$
$$
H_{n+2}(N(\mathcal{U}_{j+1}) \times I \underset{s_j}{\cup} N (\mathcal{U}_{j}), \mathbb{\Omega}^{NPD})
$$
maps
$
\{f,b\}''_j-\{f_0,b_0\}''_j
$
to
$
\{f,b\}_j-\{f_0,b_0\}_j
$. 

This completes the proof of the lemma.
\end{proof}
\begin{remark}\label{C:remark2} Transversality implies that the assembly construction defines an isomorphism between 
$\mathbb{\Omega}^{NTOP}$
homology groups 
and 
$\mathbb{\Omega}^{NTOP}$
bordism groups. This is not true for $\mathbb{\Omega}^{NPD}$. 
\end{remark}
 We now continue with the proof of Theorem~\ref{C:Theorem 5.1}.
 Denote by 
 $$
  { \langle f,b \rangle-\langle f_0,b_0 \rangle}
 \in 
 H^{st}_{n}(\underset{j}{\varprojlim} N(\mathcal U_j), \mathbb{L}^{+})
 $$ 
 the image of 
  $
  \{f,b\}-\{f_0,b_0\}
 $
 under the composition
 $$
  H^{lf}_{n+2} (F_0,\mathbb{\Omega}^{NPD})
  \cong
   H^{st}_{n+1} (CF, \underset{j}{\varprojlim} N(\mathcal U_j) \cup \{*\},\mathbb{L}^{+})
   \xrightarrow{\partial'_{*}}  
   H^{st}_{n}(\underset{j}{\varprojlim} N(\mathcal U_j), \mathbb{L}^{+}),
 $$
 where
 $
 \partial'_{*}
 $
 is the composition of the boundary homomorphism with the projection
 $$
 H^{st}_{n}(\underset{j}{\varprojlim} N(\mathcal U_j), \mathbb{L}^{+})\oplus H^{st}_{n}(\{*\},\mathbb{L}^{+})
   \to  
   H^{st}_{n}(\underset{j}{\varprojlim} N(\mathcal U_j), \mathbb{L}^{+}).
 $$
 
 Lemma~\ref{C:Lemma 5.3}
 now implies that
 $$
 (\restr{q}{\underset{j}{\varprojlim} N(\mathcal U_j)})_{*} 
  { (\langle f,b \rangle''-\langle f_0,b_0 \rangle'')
 =
 \langle f,b \rangle-\langle f_0,b_0\rangle,} 
 $$
 where  
  $\langle f,b \rangle''$
  and
  $\langle f_0,b_0 \rangle''$
  denote the corresponding images of 
  $\{f,b\}''$
  and
 $\{f_0,b_0\}''$, respectively.

In order to complete the proof of Theorem~\ref{C:Theorem 5.1}, we have to pass to 
$$
H^{st}_{n}(X, \mathbb{L}^{+})
$$
via the homomorphism
$$
\psi_{*}: H^{st}_{n}(\underset{j}{\varprojlim} N(\mathcal U_j), \mathbb{L}^{+})
\to
H^{st}_{n}(X, \mathbb{L}^{+}),
$$
induced by the map
$$
\psi: \underset{j}{\varprojlim} N(\mathcal U_j)
\to
X
$$
which was defined in Theorem~\ref{L:maps}.
Similarly for the map
$$
\psi'': \underset{j}{\varprojlim} N(\mathcal U''_j)
\to
X
$$

Now, observe that
$$
\restr{q}{ N(\mathcal U''_j)}=p_j
$$
and that
$
\psi''_j
$
is homotopic to 
$
\psi_j\circ p_j.
$
Hence the following diagram commutes
\begin{center}
	\begin{tikzpicture}[node distance=4cm, auto]
	  \node (LU) {};	   
	  \node (NU') [right of=LU] {$H^{st}_{*}(N(\mathcal{U}''_{j}), \mathbb{L}^{+})$};
	  \node (NU) [below of=NU'] {$H^{st}_{*}(N(\mathcal{U}_{j}), \mathbb{L}^{+})$};
	  \node (RU) [right of=NU'] {};
	  \node (X2) [node distance=2cm, below of=RU] {$H^{st}_{*}(X, \mathbb{L}^{+})$};	    
	  \draw[->, font=\small] (NU') to node {$(\restr{q}{N(\mathcal U''_j)})_{*}$} (NU);
	  \draw[->, font=\small] (NU') to node [midway, above] {$(\psi''_j)_{*}$} (X2);
	  \draw[->, font=\small] (NU) to node [midway, below]{$(\psi_j)_{*}$} (X2);
	\end{tikzpicture}
	\end{center}	 
	
It follows that the diagram below
\begin{center}
	\begin{tikzpicture}[node distance=4cm, auto]
	  \node (LU) {};	   
	  \node (NU') [right of=LU] {$H^{st}_{*}(\underset{j}{\varprojlim} N(\mathcal{U}''_{j}), \mathbb{L}^{+})$};
	  \node (NU) [below of=NU'] {$H^{st}_{*}(\underset{j}{\varprojlim} N(\mathcal{U}_{j}), \mathbb{L}^{+})$};
	  \node (RU) [right of=NU'] {};
	  \node (X2) [node distance=2cm, below of=RU] {$H^{st}_{*}(X, \mathbb{L}^{+})$};	    
	  \draw[->, font=\small] (NU') tonode {$(\restr{q}{N(\mathcal U''_j)})_{*}$} (NU);
	  \draw[->, font=\small] (NU') to node [midway, above] {$(\psi''_j)_{*}$} (X2);
	  \draw[->, font=\small] (NU) to node [midway, below]{$(\psi_j)_{*}$} (X2);
	\end{tikzpicture}
	\end{center}		
also commutes, thus we can see that
$$
\psi''_{*} (\langle f,b \rangle'' - \langle f_0,b_0 \rangle'')
 =
 \psi_{*}(\langle f,b \rangle - \langle f_0,b_0 \rangle).
$$

Analogously, 
one obtains
$$
 q':
 (CF'', F''_0, \underset{j}{\varprojlim} N(\mathcal U''_j))
 \to
 (CF', F'_0, \underset{j}{\varprojlim} N(\mathcal U'_j))
 $$
 such that
$$
\psi''_{*} (\langle f,b \rangle'' - \langle f_0,b_0 \rangle'')
 =
 \psi'_{*}(\langle f,b \rangle' - \langle f_0,b_0 \rangle').
$$

 This proves
 $$
 [f,b]'-[f_0,b_0]'
 =
 [f,b]-[f_0,b_0],
 $$
 and hence finally, completes the proof of Theorem~\ref{C:Theorem 5.1}.
 \end{proof} 
 We shall now apply our construction to the case when
  $X$ is a  manifold with simplicial structure.
 The given 
 degree one normal map
 $(f,b):M\to X$
 then decomposes into adic pieces to define an element
 $$
 \sigma^{c}_{*}(f,b)
 \in
 H_{n}(X, \mathbb{L}^{+}).
 $$
 This element is the controlled surgery obstruction of $(f,b)$
 over $Id:X\to X$ (cf. Pedersen, Quinn, and Ranicki~\cite{PeQuRa03}).
 We take
  $
 (f_0,b_0)=Id:X\to X.
 $ \\
 
\noindent
{\bf Supplement.}
In Section~\ref{SS:3.2}  we associated to given normal degree one maps 
 $$
 (f_0,b_0):M_0\to X
 \ \ \ 
 \mbox{ and}
 \ \ \  
 (f,b):M\to X,
 $$
  where $X$ is a generalized manifold, the element
 $$
  [f,b]  -  [f_0,b_0]  \in H^{st}_{n}(X,\mathbb{L}^{+}).
 $$
 The normal maps $(f_0,b_0)$ and $(f,b)$
 give rise to a normal space with boundary 
 $$M_0 \times (0,\infty)
 \ \ 
 \mbox{ and}
 \ \ 
 M \times (0,\infty).$$
 
 At this point, transversality for normal spaces (with TOP-manifold boundaries) is used to split (disassemble) the normal space,
 in order to obtain an element in the
 $\mathbb{\Omega}^{NPD}$-homology group. 
 
Actually, it belongs to the 
 $\mathbb{\Omega}^{NTOP}$-homology, 
 but we pass to 
 $\mathbb{\Omega}^{NPD}$ via $$\mathbb{\Omega}^{NTOP}\to\mathbb{\Omega}^{NPD}.$$
 A detailed splitting construction can be found in K\" uhl, Macko, and Mole~\cite[Construction 11.3, p.236]{KMM13}.
 
 If $X$ is a manifold with simplicial structure, transversality directly applies to split $(f,b)$
 and $(f_0,b_0)$
 into pieces in order
 to obtain an element
 in $H_n(X, \mathbb{L}).$
 It is now natural to take 
 $$(f_0,b_0)=Id:X\to X.$$
 
 Since 
 $Id:X\to X$ does not contribute to $\mathbb{L}$-homology, one gets an element depending on 
 $(f,b)$ which we shall denote by 
 $$
 \sigma(f,b)\in H_n(X,\mathbb{L})
 $$ 
 (this corresponds to 
 sig$_{X}^{\mathbb{L}}(f,b)$ in K\" uhl, Macko, and Mole~\cite[Definition 8.14]{KMM13}). 
 Moreover,  
 $$
 \sigma(f,b)\in H_n(X, \mathbb{L}^{+}).
 $$
 
 Since $ H_n(X,\mathbb{L})$
is the controlled surgery obstruction group,
the element
 $$\sigma(f,b)\in
 H_n(X,\mathbb{L}^{+}) \subset H_n(X,\mathbb{L})
 $$
 is
 sometimes denoted by 
$\sigma^{c}(f,b)$.

 The reason is that the 0-dimensional components come from
 $$
 f^{-1}(D(\sigma, X)) \to D(\sigma, X),
 $$
 where $\sigma \prec X$ runs through the $n$-simplices of $X$
  and $D(\sigma, X)$ is its dual with respect to a subdivision $X'$ of $X$. 
 Hence
  $D(\sigma, X)$ is a point $x\in X$ and by transversality,
$$f^{-1}(D(\sigma, X))=\{\pm y_1,...,\pm y_k\}\subset M.$$
 
 Since $f$ has degree one, it is equivalent to $y\to x$, which is the tri\-vial object.
 (We have also addressed such questions in Hegenbarth and Repov\v{s}~\cite[Lemma 2.1]{HeRe19}.)
 
 If $X$ is only a generalized manifold, this leads to the so-called 0-dimensional signature of $f$.
 This is misleading, since it is the signature obstruction of a $4k$-dimensional surgery problem, which is ``moved'' to $\pi_0(\mathbb{L})=L_0$ by periodicity of $\mathbb{L}$ (cf. Hegenbarth and Repov\v{s}~\cite[p. 79]{HeRe19}).\\ 

The aim of the next theorem is to show that for a given degree one normal map
$(f,b):M\to X$, 
where $X$ is a
 manifold
 with simplicial structure, the construction via normal spaces from Section~\ref{SS:3.2}, gives an element which coincides with the element
$\sigma(f,b)$. 
 \begin{theorem}\label{C:Proposition 5.4}
 The controlled surgery obstruction of
 $
 (f,b):M\to X
 $
 coincides with
 $
 [f,b]-[f_0,b_0].
 $
 \end{theorem}
 \begin{proof}
 Choose a sequence
 $\{ \mathcal{U}_j\}$
 of coverings 
 of $X$
 as above.
 Since $X$ is a
  manifold
  with simplicial structure,
 we can define
 $$
 \overline{\{f,b\}}
 -
 \overline{\{f_0,b_0\}}
 \in
 H^{lf}_{n+2} (X\times (0,\infty),\mathbb{\Omega}^{NPD}).
 $$
 Here, 
 $\overline{\{f,b\}}$
 denotes the normal space, defined by the mapping cylinder of the map
 $$
 (f\times Id, b\times Id): M\times (0,\infty)\to X\times (0,\infty),
 $$
 and similarly, for 
 $\overline{\{f_0,b_0\}}.$
 \\
 Now,
 $ $$
 \overline{\{f,b\}}
 -
 \overline{\{f_0,b_0\}}$
  maps under
  the induced map
  $$
 \Gamma: X\times (0,\infty) \to F_0
 $$
 to
  $$
 {\{f,b\}}
 -
 {\{f_0,b_0\}}
 \in
 H^{lf}_{n+2} (F_0,\mathbb{\Omega}^{NPD}).
 $$
 Under the composition
 $$
  H^{lf}_{n+2} (X\times (0,\infty),\mathbb{\Omega}^{NPD})
  \cong
   H^{lf}_{n+1} (X\times (0,\infty),\mathbb{L}^{+})
  \cong  
  $$    
   $$   
    H^{st}_{n+1} (X\times [0,\infty]/(X\times \{\infty \}), X\cup \{*\}, \mathbb{L}^{+})
  \xrightarrow{\partial'_{*}} 
H^{st}_{n} (X, \mathbb{L}^{+}),
 $$ 
$
 \overline{\{f,b\}}
 -
 \overline{\{f_0,b_0\}}
 $
 maps to
 $
 \sigma^{c}_{*}(f,b).
 $
 This is because 
 $
 \overline{\{f,b\}}
 -
 \overline{\{f_0,b_0\}}
 $
 is represented by the mapping cylinders of
 $$
 (f\times Id, b\times Id)
 \ \ 
 \mbox{and}
 \ \ 
 (f_0\times Id, b_0\times Id).
 $$ 
 
 The latter one does not contribute to 
 $\mathbb{L}$-homology
 because we have chosen
 $(f_0,b_0)=(Id,Id)$.
 Under the composition it therefore goes to the element defined by splitting
  $(f,b):M\to X$, i.e. to $\sigma(f,b)$.
 
 Here,
$$
 X\times [0,\infty]/(X\times \{\infty \})
 $$
 is the completion of 
 $
 X\times (0,\infty)
 $ 
 obtained as the inverse limit, similarly as
 $$
 CF=\underset{l}{\varprojlim} F_l
 $$
 (cf. Chapter 3).
 However, under 
  $$
  \Gamma_{*}:
  H^{lf}_{n+2} (X\times (0,\infty),\mathbb{\Omega}^{NPD})
  \to
  H^{lf}_{n+2} (F_0,\mathbb{\Omega}^{NPD})
 $$ 
 the difference
 $
  \overline{\{f,b\}}
 -
 \overline{\{f_0,b_0\}}
 $
 maps to 
  $
  \{f,b\}
 -
 \{f_0,b_0\}.
 $
 
 Consider now (using previous notations)
 $$
 X\times I 
  \xrightarrow{\Gamma_j} 
  N(\mathcal{U}_{j+1})\times I
 \underset{s_j}{\cup}
   N(\mathcal{U}_{j})
  \xrightarrow{\overline{\Lambda}_j}
  X\times I,
 $$
 where
 $$
 {\overline{\Lambda}_j}(u,t)=(\Lambda_j (u),t)
 $$
 so 
 $$
 {\overline{\Lambda}_j} \circ
 \Gamma_j
 (x,0)=
 ((\psi_{j+1}\circ \varphi_{j+1})(x),0)
 $$
 and
 $$
  {\overline{\Lambda}_j} \circ
 \Gamma_j
 (x,1)=
 ((\psi_{j}\circ \varphi_{j})(x),1).
 $$
 
 Since
 $
 \psi_{k}\circ \varphi_{k} \sim Id_X
 $
 we can use these homotopies to glue the maps and obtain 
 $$
 X\times \mathbb{R}_{+} 
  \xrightarrow{\Gamma} 
  F_0
  \xrightarrow{\overline{\Lambda}}
  X\times \mathbb{R}_{+}, 
 $$ 
 restricting to
 $$
 X\times \{0\} 
  \xrightarrow{\varphi} 
   \underset{j}{\varprojlim} N(\mathcal U_j)
  \xrightarrow{\psi}
  X\times \{0\}, 
 $$
 i.e. we get maps
 $$
 X\times [0,\infty]/(X\times \{\infty \})
 \to
 CF
 \to
  X\times [0,\infty]/(X\times \{\infty \}).
 $$
 
 Therefore $\Gamma$
 induces a morphism
 between the sequences  
   $$
  H^{lf}_{n+2} (X\times [0,\infty),\mathbb{\Omega}^{NPD})
  \cong
  H^{lf}_{n+1} (X\times [0,\infty),\mathbb{L}^{+})  \cong
  $$ 
  $$ 
  H^{st}_{n+1} (X\times [0,\infty]/(X\times \{\infty \}), X\cup \{*\},\mathbb{L}^{+})  
  \xrightarrow{\partial'_{*}}
  H^{st}_{n} (X,\mathbb{L}^{+}),
 $$
  and
   $$
  H^{lf}_{n+2} (F_0,\mathbb{\Omega}^{NPD})
  \cong
  H^{lf}_{n+1} (F_0,\mathbb{L}^{+})  
  \cong 
  $$
  $$
  H^{st}_{n+1} (CF, \underset{j}{\varprojlim} N(\mathcal U_j) \cup \{*\},\mathbb{L}^{+})  
  \xrightarrow{\partial'_{*}}
  H^{st}_{n} (\underset{j}{\varprojlim} N(\mathcal U_j),\mathbb{L}^{+}).
 $$
 
 It follows that  
 $$
 \varphi_{*}(\sigma^{c}_{*}(f,b))= \langle f,b \rangle - \langle f_0,b_0 \rangle
 \in
 H^{st}_{n} (\underset{j}{\varprojlim} N(\mathcal U_j),\mathbb{L}^{+}).
 $$ 
 Since
 $
 \psi \circ \varphi \sim Id,
 $
 we can conclude that 
 $$
 [f,b]-[f_0,b_0]=\psi_{*}(\langle f,b \rangle - \langle f_0,b_0 \rangle)=
  \psi_{*}(\varphi_{*}( \sigma^{c}_{*}(f,b)))= \sigma^{c}_{*}(f,b).
 $$  
 \end{proof}
 We shall finish this section
 with some remarks on the map $t$. 
 In the PL manifold case there is an
${\mathbb{L}}^{\bullet}$-orientation
$$
\mathcal{U}_{{\mathbb{L}}^{\bullet}}\in
\overline{H}^{m-n}(T(\nu_{X}), {\mathbb{L}}^{\bullet}),
$$
where
${\mathbb{L}}^{\bullet}$
is the symmetric
$\mathbb{L}$-spectrum.
Furthermore,
${\mathbb{L}}^{\bullet}$
is a ring spectrum and 
$\mathbb{L}^{+}$
is an 
${\mathbb{L}}^{\bullet}$-module spectrum, and the cup product
$$
. \cup \  \mathcal{U}_{{\mathbb{L}}^{\bullet}}:
[X,G/TOP]=H^{\circ}(X,\mathbb{L}^{+})
\to
\overline{H}^{m-n}(T(\nu_{X}), \mathbb{L}^{+})
$$
is an isomorphism. 
Here we are assuming that 
$X\subset \mathbb{R}^m$.

The difference between 
$(f,b):M\to X$
and
$(f_0,b_0):X\to X$
defines a map
$$\mathcal{N}(X)\to [X,G/TOP].$$
Combining with the Alexander-Spanier duality
$$
\overline{H}^{m-n}(T(\nu_{X}), \mathbb{L}^{+})
\cong
H_{n}(X, \mathbb{L}^{+}),
$$
we obtain a bijective map
$$\mathcal{N}(X)\to 
H_{n}(X, \mathbb{L}^{+}).$$
This is the map $t$ 
 (cf. Ranicki~\cite[Chapter 17, pp.191-193]{Ranicki92}).
 
 In the case of a generalized manifold we can embed $X$  
into $\mathbb{R}^m$ with a cylindrical neighborhood, also obtaining an isomorphism 
$$
\overline{H}^{m-n}(T(\nu_{X}), \mathbb{L}^{+})
\cong
H^{st}_{n}(X, \mathbb{L}^{+}).
$$ 
Let
$$
N=\partial N \times I \underset{p}{\cup} X$$
be a mapping cylinder neighborhood of
$X\subset S^{m+1}$. It can be used to prove the following fact. 
\begin{theorem}\label{C:Proposition 5.5.}
There exists an
${\mathbb{L}}^{\bullet}$-orientation 
$$
\mathcal{U} \in
H^{m+1-n}(N, \partial N,{\mathbb{L}}^{\bullet}).
$$
and an isomorphism
$$
.\cup \mathcal{U}:
H^{0}(X,\mathbb{L}^{+}) 
\xrightarrow{\cong}
H^{m+1-n}(N, \partial N, \mathbb{L}^{+}).
$$
\end{theorem} 
With this theorem one obtains the following isomorphisms:
$$
H^{0}(X,\mathbb{L}^{+}) 
\cong
H^{m+1-n}(N, \partial N, \mathbb{L}^{+})
\cong
H^{m+1-n}(S^{m+1}, S^{m+1}\setminus N, \mathbb{L}^{+})
\cong
$$
$$
H^{m+1-n}(S^{m+1}, S^{m+1}\setminus X, \mathbb{L}^{+})
\cong
$$
$$
\overline{H}^{m-n}(S^{m+1}\setminus X, \mathbb{L}^{+})
\cong
H^{st}_{n}(X, \mathbb{L}^{+}).
$$
The last isomorphism is the Steenrod duality 
(cf. 
Kahn, Kaminker and Schochet~\cite[Theorem B]{KaKaSc77},
there one must take the reduced 
$\mathbb{L}^{+}$-homology).

We shall omit the proof of Theorem~\ref{C:Proposition 5.5.} because it is not obvious that the composition
$$
\mathcal{N}(X)
\to
H^{0}(X,\mathbb{L}^{+}) 
\to
H^{st}_{n}(X, \mathbb{L}^{+})
$$
coincides with the association
$$(f,b)\to [f,b]-[f_0,b_0].$$
This will be included in our future paper. 
\section{Discussion}\label{SS:remarks}
\paragraph{{\bf  (I)}}\label{SS:3.3a}
The homotopy groups of the spectrum $\mathbb{L}^+$ are the Wall groups of the trivial group, i.e. 
$$\pi_n (\mathbb{L}^+) \cong L_n (1) \cong L_n \ \ {\hbox{\rm{ for every}}}  \ \ n\geq 1.$$
 Since the simplicial complex 
 $$N (\mathcal{U}_{1}) \times I \underset{s_0}{\cup} N (\mathcal{U}_{0})$$ is contractible,
  we have 
$$H^{st}_{n} (N (\mathcal{U}_{1}) \times I \underset{s}{\cup} N (\mathcal{U}_{0}), \mathbb{L}^+) \cong H_n (N (\mathcal{U}_{1}) \times I \underset{s_0}{\cup} N (\mathcal{U}_{0}), \mathbb{L}^+) \cong $$
$$\cong H_n (\{b_0\}, \mathbb{L}^+) \cong L_n.$$

Therefore the above mentioned homology boundary homomorphism  is 
 $$
 H^{st}_{n+1} (CF, \underset{j}{\varprojlim} N (\mathcal{U}_{j})  {\amalg} \{b_0\}, \mathbb{L}^+) \rightarrow H^{st}_{n} ( \underset{j}{\varprojlim} N (\mathcal{U}_{j}), \mathbb{L}^+) \oplus L_n.
 $$

The component in $L_n$ is the surgery obstruction of 
$$(f,b) = (M^n \rightarrow X^n, \ \ \nu_{M^n} \rightarrow \xi)$$
mapped to $L_n$ under 
 $$L_n (\pi_1 (X^n)) \rightarrow L_n (1) \cong L_n,$$
 where the morphism is induced
 by
  $X^n \rightarrow \{*\}$.\\
  
\paragraph{{\bf  (II)}}\label{SS:3.3b}
We have used the map 
$$\psi_* : H^{st}_{n} ( \underset{j}{\varprojlim} N (\mathcal{U}_{j}), \mathbb{L}^+) \rightarrow H^{st}_{n} ( X^n, \mathbb{L}^+)$$
 to obtain our element 
 $$\left[ f,b\right] \in H^{st}_{n} ( X^n, \mathbb{L}^+).$$
  We did not need that it is an isomorphism. 
 
 In fact, the relation between $\underset{j}{\varprojlim} N (\mathcal{U}_{j})$ and $X^n$ seems to be insufficiently documented. It was claimed in 
 Milnor~\cite[Lemma 2]{Milnor95} that they are identical. It was
 also
  asserted
 in Ferry~\cite[Footnote, p. 156]{Ferry95} that they are strongly shape equivalent. 
 
 To this end,
  we state the following theorem.
  First, recall that given $\varepsilon >0$,
  a map 
  $f:X\to Y$ of metric spaces $X$ and $Y$ is called
   an $\varepsilon$-map if for every $y \in Y$, $\hbox{\rm{diam}}( f^{-1}(y)) < \varepsilon$.
\begin{theorem}\label{6.1}
The maps 
$$\varphi_j : X^n \rightarrow N (\mathcal{U}_{j+1}) \times I \underset{s_j}{\cup} N (\mathcal{U}_{j})$$
 fit together to produce
 the map $$\varphi: X^n \rightarrow \underset{j}{\varprojlim} N (\mathcal{U}_{j}),$$ which is an $\varepsilon$-map onto the image of $\varphi$ for all $\varepsilon > 0$.
\end{theorem}
\begin{proof}
The maps $\Gamma_j$ can be glued to get maps 
$$X^n_l = X^n \times \left[ 0, l+1 \right] \rightarrow F^\circ_l,$$ such that 
the diagram
\begin{center}
\begin{tikzpicture}[node distance=2.5cm, auto]
  \node (Xl) {$X^n_l$};
  \node (Fl) [right of=Xl] {$F^\circ_l$};
  \node (Xl1) [below of=Xl] {$X^n_{l-1}$};
  \node (Fl1) [below of=Fl] {$F^\circ_{l-1}$};
  \draw[->, font=\small] (Xl) to node {} (Fl);
  \draw[->, font=\small] (Xl) to node {$pr$} (Xl1);
  \draw[->, font=\small] (Fl) to node {$r_l$} (Fl1);
  \draw[->, font=\small] (Xl1) to node {} (Fl1);
\end{tikzpicture}
\end{center}
commutes. 

Hence
 we get a map 
$$X^n \times \left[ 0, \infty\right] \rightarrow \underset{l}{\varprojlim} F^\circ_l$$
 which restricts to  
$$\varphi: X^n \times \{\infty\} \rightarrow \underset{j}{\varprojlim} N (\mathcal{U}_{j}).$$

If now
$$p_l :  \underset{j}{\varprojlim} N (\mathcal{U}_{j}) \rightarrow N (\mathcal{U}_{l})$$ is the projection,
 then $p_l \circ \varphi = \varphi_l$.

Let $x\in {\hbox{\rm Im}} \varphi$. Then $$x_l = p_l (x) \in N (\mathcal{U}_{l})$$ belongs to some ${\hbox{\rm st}} (\left< U\right>)$,  for some vertex
$\left< U\right>\in N (\mathcal{U}_{l})$, where $U \in \mathcal{U}_{l}$.

 Therefore 
$$\varphi^{-1} (x)\subset \varphi^{-1}_l (x_l) \subset U$$
 (cf. 
 Dugundji~\cite[Theorem~5.4, Chapter~VIII]{Dugundji73}). Hence
  $${\hbox{\rm diam}} ( \varphi^{-1}(x)) \leq {\hbox{\rm mesh}}(\mathcal{U}_{l})$$
  and since ${\hbox{\rm mesh}}(\mathcal{U}_{j}) \rightarrow 0$ for $j \rightarrow \infty$, the assertion follows.
\end{proof}
\begin{remark}
It would be interesting to know if the Bing Shrinking Criterion (cf. Marin and Visetti~\cite{bing}) can be applied to improve Theorem~\ref{6.1}.
\end{remark}
\section*{Acknowledgements}
This research was supported by the Slovenian Research Agency grants P1-0292, J1-8131, N1-0064,  N1-0083, and N1-0114.  
We very gratefully acknowledge the referee for several important comments and suggestions
which have considerably improved the presentation.


\begin{thebibliography}{999} 
\bibitem{BFMW96} J.L. Bryant, S. Ferry, W. Mio, and S. Weinberger, 
Topology of homology manifolds, \textit{Ann. of Math.} 143 (2) (1996), 435-467. MR 97b:57017
\bibitem{CaHeRe16} A. Cavicchioli, F. Hegenbarth, and D. Repov\v s,
 \textit{Higher-Dimensional Generalized Manifolds: Surgery and Constructions,} 
 EMS Series of Lectures in Mathematics {\bf 23}, European Mathematical Society, Z\" urich, 2016. 
 MR 3558558 
\bibitem{Dugundji73}  J. Dugundji, \textit{Topology},  Allyn and Bacon, Boston, 1973.   MR 33\#1824
\bibitem{Ferry92} S.C. Ferry, Geometric Topology Notes,   Unpubl. manuscr., Rutgers Univ., Piscataway, NJ, 2008.
{https://sites.math.rutgers.edu/$\sim$sferry/ps/geotop.pdf}
\bibitem{Ferry95} S.C. Ferry, Remarks on Steenrod homology, 
 \textit{Novikov Conjectures, Index Theorems, and Rigidity, Vol. 2},
 S. Ferry, A. Ranicki, and J. Rosenberg, Eds.,
  London Math. Soc. Lecture Note Ser. 227,
  Cambridge University Press,
   Cambridge,
   1995, pp. 148-166. MR 98d:55004
\bibitem{Ferry10} S.C. Ferry, Epsilon-delta surgery over $\mathbb{Z}$,
 \textit{Geom. Dedicata} 148 (2010), 71-101. MR 2011m:57030
\bibitem{Fe-Ped} S.C. Ferry and E.K. Pedersen, Epsilon surgery theory,
 \textit{Novikov Conjecture, Index Theorem and Rigidity, Vol. 2},  
 S. Ferry, A. Ranicki, and J. Rosenberg, Eds.,
 London Math. Soc., Lecture Notes Series 227, London, 1995, pp. 167-226. MR 97g:57044
\bibitem{HaVo93} J.C. Hausmann and P. Vogel, \textit{Geometry on Poincar\'{e} Spaces},
 Math. Notes 41, Princeton University Press, Princeton, NJ, 1993. MR 95a:57042
\bibitem{HeRe14} F. Hegenbarth and D. Repov\v{s},
 Controlled homotopy equivalences on structure sets of manifolds,
  \textit{Proc. Amer. Math. Soc.} 142 (2014), 3987-3999. MR 3251739  
  \bibitem{HeRe19} F. Hegenbarth and D. Repov\v{s},
 Controlled surgery and $\mathbb{L}$-homology.  
  \textit{Mediterr. J. Math.} 16 (2019), no. 3, Art. 79, 22 pp. MR 3945264 
\bibitem{Hu65}  S.T. Hu, \textit{Theory of Retracts},
  Wayne State Univ. Press,  Detroit,  1965. MR 31\#6202
\bibitem{Jones73} L. Jones, Patch spaces: a geometric representation for Poincar\'{e} spaces,
 \textit{Ann. of Math. (2)} 97 (1973), 306-343. MR 0315720;
{ Corrections}, \textit{ Ann. of Math. (2)} 102 (1975), 183-185. MR 0391108  
\bibitem{KaKaSc77} D.S. Kahn, J. Kaminker, and C. Schochet,
 Generalized homology theories on compact metric spaces,
  \textit{Michigan Math. J.} 24 (2) (1977), 203-224. MR 0474274
\bibitem{KMM13}
P. K\" uhl, T. Macko, and A. Mole,
The total surgery obstruction revisited, 
\textit{M\" unster J. Math.} 6 (2013), 181-269. MR3148212  
\bibitem{Levitt72} N. Levitt, Poincar\'{e} duality cobordism,
\textit{Ann. of Math. (2)} 96 (1972), 211-244. MR 0314059
\bibitem{bing}
  A. Marin and Y.M. Visetti, 
  A general proof of Bing's shrinkability criterion,  
  \textit{Proc. Amer. Math. Soc.}  53 (2)  (1975),  501-507. MR 52\#9156
\bibitem{Milnor95} J. Milnor, On the Steenrod homology theory,
 \textit{Novikov Conjectures, Index Theorems, and Rigidity, Vol. 1},
 S. Ferry, A. Ranicki, and J. Rosenberg, Eds.,
  London Math. Soc. Lecture Note Ser. 226, Cambridge University Press,
  Cambridge, 1995, pp. 79-96. MR 98d:55005
\bibitem{Nicas82} A. Nicas, Induction theorems for groups of homotopy manifold structures,
 \textit{Mem. Amer. Math. Soc.} 39 (267), 
 Amer. Math. Soc., Providence, RI, 1982. MR 83i:57026
\bibitem{PeQuRa03} E.K. Pedersen, F. Quinn, and A. Ranicki, Controlled surgery with trivial local fundamental groups,
\textit{High-Dimensional Manifold Topology},
F.T. Farrell and W. L\" uck, Eds., World Sci. Publishing, River Edge, N.J., 2003, pp. 421-426. MR 2005e:57077
\bibitem{Quinn69}  F.S. Quinn, \textit{A Geometric Formulation of Surgery},
  Doctoral Dissertation,  Princeton University,  Princeton,  1969. MR 2619602
\bibitem{Quinn72}  F.S. Quinn, Surgery on Poincar\'{e} and normal spaces,
  \textit{Bull. Amer. Math. Soc.}  78 (1972), 262-267. MR 0296955
\bibitem{Qui83} F.S. Quinn, Resolutions of homology manifolds and the topological characterization of manifolds,
 \textit{Invent. Math.} 72 (2) (1983), 267-284; Corrigendum,  \textit{Invent. Math.} 85 (3) (1986), 653. 
 MR 85b:57023, MR 87g:57031
\bibitem{Qui87} F.S. Quinn, An obstruction to the resolution of homology manifolds,
 \textit{Michigan Math. J.} 34 (2) (1987), 285-291. MR 88j:57016
\bibitem{Ranicki79} A.A. Ranicki, The total surgery obstruction,
\textit{Proc. Alg. Topol. Conf. Aarhus 1978}, Lect. Notes Math., 
Springer-Verlag, Berlin 763 (1979), pp. 275-316. MR 81e:57034
\bibitem{Ranicki92} A.A. Ranicki, \textit{Algebraic $L$-Theory and Topological Manifolds},
 Cambridge Tracts in Math. 102, Cambridge Univ. Press, Cambridge, 1992. MR 94i:57051 
\bibitem{Yamasaki87} M. Yamasaki, $L$-groups of crystallographic groups,
 \textit{Invent. Math. } 88 (3) (1987), 571-602. MR 88c:57017
\bibitem{Wal67}  C.T.C. Wall, Poincar\'e complexes, I. \textit{Ann. of Math.}    86 (2)  (1967),  213-245. MR 36\#880
\bibitem{Wal99}  C.T.C. Wall, \textit{Surgery on Compact Manifolds},
 Second Ed., 
 Edited and with a foreword by A.A. Ranicki,
  Mathematical Surveys and Monographs 69, American Mathematical Society, Providence, RI, 1999. MR1687388
\bibitem{Wei02} S. Weinberger,  Homology manifolds,
 \textit{Handbook of Geometric Topology},  North-Holland, Amsterdam, 2002, 1085-1102. MR  2003b:57032 
\bibitem{West77} J.E. West,  Mapping Hilbert cube manifolds to ANR's: A solution of a conjecture of Borsuk,
\textit{ Ann. of Math. } (2) 106 (1977),  1-18. MR 56 \#9534
\end{thebibliography}
\end{document}